\documentclass{article}

\usepackage{amsmath}
\usepackage{amsfonts}
\usepackage{amsthm}
\usepackage{amssymb}
\usepackage{mathtools}
\usepackage{graphicx}
\usepackage{color}
\usepackage{verbatim}
\usepackage[hidelinks]{hyperref}
\usepackage{cleveref}
\usepackage[utf8]{inputenc}
\usepackage[backend=biber,url=false,maxbibnames=5]{biblatex}
\usepackage{enumitem}

\newtheorem{theorem}{Theorem}

\newtheorem{corollary}[theorem]{Corollary}
\newtheorem{lemma}[theorem]{Lemma}

\theoremstyle{definition}

\renewcommand{\phi}{\varphi}

\renewcommand{\P}{\mathcal{P}}

\newcommand{\s}{\subseteq}
\newcommand{\sm}{\setminus}

\newcommand{\lf}{\lfloor}
\newcommand{\rf}{\rfloor}

\title{$m$-Modular Wythoff}
\author{Tanya Khovanova \and Nelson Niu}
\date{}

\begin{document}

  \maketitle

  \begin{abstract}
    We introduce a variant of Wythoff's Game that we call $m$-Modular Wythoff's Game.
    In the original Wythoff's Game, players can take a positive number of tokens from one pile, or they can take a positive number of tokens from both piles if the number of tokens they take from the first pile is equal to the number of tokens they take from the second.
    In our variant, we weaken this equality condition to one of equivalence modulo $m$.
    We characterize the P-positions of our $m$-Modular variant as a finite subset of the P-positions of the known P-positions of the original Wythoff's Game.
  \end{abstract}

  \section{Introduction}

  The Game of Nim forms the foundation of the mathematical study of two-player impartial games.
  In his landmark paper \cite{Nim}, Bouton \textit{solved} the Game of Nim by providing a winning strategy when one exists and proving when one did not, effectively founding the field of Combinatorial Game Theory.

  One of the most famous variants of Nim is Wythoff's Game, introduced and solved by Wythoff in \cite{Wy}.
  Both Nim and Wythoff's Game are examples of \textit{invariant games}, defined in \cite{Du}: games whose set of valid moves is independent of the position from which the moves are played.

  Several further variants of Nim and Wythoff have been studied, but few have moves based on modular congruence.
  One such example is Modular Nim, studied in \cite{Fr}.
  The first author recently explored a different modular variant of Nim in \cite{KS}.
  Dubbed $m$-Modular Nim, it is an invariant game in which moves predicated on modular congruence are added to the traditional Nim moves.

  In this paper, we introduce and solve an analogous invariant modular extension of Wythoff's Game.
  Here we mean \textit{extension} in the sense of \cite{ext}: we expand the set of moves available in Wythoff's Game without removing any.
  In \cite{ext}, Duchêne et al.\ examine extensions of Wythoff's Game that preserve the P-positions of the original game.
  In contrast, we will show that our extension---which we call $m$-Modular Wythoff's Game---restricts the P-positions to a finite subset of those of the original Wythoff's Game.

  We begin in \Cref{sec:WG} by reviewing the rules and P-positions of Wythoff's Game and introducing notation.
  Then in \Cref{sec:mWG}, we define $m$-Modular Wythoff's Game, where in addition to the usual Wythoff moves, we allow players to remove tokens from both piles as long as the number of tokens removed from one pile is equivalent modulo $m$ to the number of tokens removed from the other.
  We proceed to compute the P-positions of $m$-Modular Wythoff for small values of $m$, showing that these positions are a subset of the P-positions of Wythoff's Game.

  In \Cref{sec:main}, we prove that the number of P-positions of $m$-Modular Wythoff is always finite.
  We do so by explicitly characterizing the P-positions of $m$-Modular Wythoff: there are $2\lfloor m/\phi \rfloor + 1$ of them, where $\phi\coloneqq\frac{1+\sqrt{5}}{2}$ is the golden ratio, and they form a subset of the P-positions of Wythoff's Game.
  Finally, we suggest some further directions of research in \Cref{sec:fd}.

  \subsection*{Acknowledgments}
  The authors are grateful to the MIT-PRIMES program for supporting this research, as well as to several anonymous reviewers for their valued feedback.

  \section{Wythoff's Game}\label{sec:WG}

  Here we review the rules and winning strategies of Wythoff's Game from \cite{Wy}.

  \subsection{Rules of Wythoff's Game}

  Let us recall the rules of Wythoff's Game from \cite{Wy}.
  The game is played with two piles of tokens.
  Two players take turns making moves, and two types of moves are allowed:
  \begin{enumerate}[label=\Roman*.]
    \item A player may take any positive number of tokens from any one pile. (Note that these are the same moves available in two-pile Nim.)
    \item A player may take the same positive number of tokens from both piles.
  \end{enumerate}
  The loser is the player who cannot move, i.e.\ the player whose turn it is when both piles are empty.

  A position of Wythoff's Game (or any other two-pile game) can be denoted by ordered pairs of nonnegative integers $(p_1,p_2)$, indicating the number of tokens in each pile.
  We say that position $p=(p_1,p_2)$ \textit{dominates} position $q=(q_1,q_2)$, writing $p\succeq q$, if $p_1-q_1$ and $p_2-q_2$ are both nonnegative.
  Moreover, we say that $p$ \textit{strictly dominates} $q$, writing $p\succ q$, if $p$ dominates $q$ and $p \neq q$.

  \subsection{P-positions of Wythoff's Game} \label{sec:ppos-wyt}

  In any combinatorial game, a \textit{P-position} is a position from which the \textit{previous} player has a winning strategy.
  Note that if a player who cannot move loses, all terminal positions are P-positions.
  On the other hand, an \textit{N-position} is a position from which the \textit{next} player has a winning strategy.
  Any move from a P-position goes to an N-position; and from any N-position, there exists a move to a P-position.

  Every game we discuss eventually terminates; in such a game, the P-positions and N-positions partition all of the game's positions.
  Finding the winning strategies of these games then amounts to computing their P-positions.

  The set of P-positions of Wythoff's Game, which we denote by $\P$, is characterized in \cite{Wy}.
  Before we describe $\P$ explicitly, we introduce the following notation.

  Define a set $\P_i$ for each nonnegative integer $i$ recursively as follows:
  \begin{equation} \label{eq.p-i-def}
    \begin{split}
      \P_0 &\coloneqq \{(0, 0)\} \\
      \text{and} \qquad \P_i &\coloneqq \P_{i-1} \cup \{(c_i, c_i+i), (c_i+i, c_i)\},
    \end{split}
  \end{equation}
  where $c_i$ is the smallest positive integer not already part of any ordered pair in $\P_{i-1}$. So $c_1=1$, making
  \[
  \P_1 = \{(0, 0), (1, 2), (2, 1)\};
  \]
  then $c_2=3$, making
  \[
  \P_2 = \{(0, 0), (1, 2), (2, 1), (3, 5), (5, 3)\};
  \]
  and so on.

  In \cite{Wy}, Wythoff shows that the P-positions of his game are given by the union of the nested sequence $\P_0\s\P_1\s\P_2\s\cdots$; that is,
  \[
  \P = \bigcup_{i=0}^{\infty} \P_i.
  \]
  Furthermore, it turns out that these P-positions are closely related to the golden ratio $\phi\coloneqq\frac{1+\sqrt{5}}{2}$, the positive number satisfying $\phi^2=\phi+1$.
  In particular, the positions $(c_i,c_i+i)$ and $(c_i+i,c_i)$ comprising each $\P_i\sm\P_{i-1}$ are given by
  \begin{equation} \label{eq.ci}
    \begin{split}
      (c_i,c_i+i)&=(\lf i\phi\rf, \lf i\phi\rf + i)=(\lf i\phi\rf, \lf i\phi^2\rf) \\
      \text{and} \qquad (c_i+i,c_i)&=(\lf i\phi\rf + i, \lf i\phi\rf)=(\lf i\phi^2\rf, \lf i\phi\rf).
    \end{split}
  \end{equation}
  The sequence
  \[
  (c_i)_{i=1}^\infty=(\lf\phi\rf,\lf2\phi\rf,\lf3\phi\rf,\ldots)
  \]
  is called the \textit{lower Wythoff sequence}, consisting of the \textit{lower Wythoff numbers}, while the sequence
  \[
  (c_i+i)_{i=1}^\infty=(\lf\phi^2\rf,\lf2\phi^2\rf,\lf3\phi^2\rf,\ldots)
  \]
  is the \textit{upper Wythoff sequence}, consisting of the \textit{upper Wythoff numbers}.
  The P-positions of Wythoff's game thus consist of the terminal position $(0,0)$ along with pairs of corresponding lower and upper Wythoff numbers.

  \section{Introducing \texorpdfstring{$m$}{m}-Modular Wythoff}\label{sec:mWG}

  We now describe the rules of \textit{$m$-Modular Wythoff's Game}, our extension of Wythoff's Game, for a positive integer $m$.
  Like the original Wythoff, $m$-Modular Wythoff is played with two piles of tokens, denoted by an ordered pair of nonnegative integers $(p_1,p_2)$, indicating the number of tokens in each pile.
  Two players take turns making moves until no tokens remain---the terminal position is $(0,0)$ reached, and the player whose turn is next loses.

  Two types of moves are allowed:
  \begin{enumerate}[label=\Roman*.]
    \item A player may take any positive number of tokens from any one pile. (Again, these are the same moves available in two-pile Nim.)
    \item A player may take a positive number of tokens from both piles, given that the difference between the number of tokens taken from each pile is divisible by $m$.
  \end{enumerate}
  We call these \textit{Type I} and \textit{Type II} moves, respectively.
  Note that the moves in Wythoff's Game are a subset of the moves in our $m$-Modular variant: $m$-Modular Wythoff is an extension of the original Wythoff.

  \subsection{Example: 2-Modular Wythoff}\label{sec:ex2}

  As an initial example, we compute the P-positions of $m$-Modular Wythoff in the case of $m=2$ via a standard argument.
  We may reason as follows.
  \begin{enumerate}
    \item The terminal position $(0,0)$ is a P-position.
    \item Positions one move away from $(0,0)$ are N-positions.
    \begin{enumerate}[label=\Roman*.]
      \item The positions that are a Type I move away from $(0,0)$ are those in which exactly one of the coordinates is $0$; these are therefore N-positions.
      \item The positions that are a Type II move away from $(0,0)$ are those with two positive coordinates of the same parity; these are therefore N-positions as well.
    \end{enumerate}
    \item From the position $(1,2)$, Type I moves can lead to $(0,2),(1,1),$ or $(1,0)$, while the only Type II move leads to $(0,1)$.
    All four of these are N-positions, so $(1,2)$ is a P-position.
    By symmetry, $(2,1)$ is a P-position as well.
    \item It remains to consider positions with positive coordinates of opposite parity not equal to $(1,2)$ or $(2,1)$.
    \begin{enumerate}[label=\Roman*.]
      \item If either coordinate of such a position is equal to $1$, then the other coordinate must be $4$ or greater, so a Type I move will bring it to either $(1,2)$ or $(2,1)$.
      Hence the position is an N-position.
      \item Otherwise, both coordinates of such a position are greater than $1$, so the position strictly dominates $(1,2)$; moreover, the coordinates have opposite parity, so a Type II move will bring it to $(1,2)$.
      Hence the position is also an N-position.
    \end{enumerate}
  \end{enumerate}
  Therefore, the only P-positions of $2$-Modular Wythoff are $(0,0),(1,2),$ and $(2,1)$; all the rest are N-positions.
  Note that the set of these positions is exactly the set $\P_1$ defined in \eqref{eq.p-i-def}.

  \subsection{More examples: \texorpdfstring{$m$}{m}-Modular Wythoff for \texorpdfstring{$m\geq5$}{m ≥ 5}}

  Following our example from \Cref{sec:ex2}, we may manually compute the P-positions of $m$-Modular Wythoff for small values of $m$.
  The results are listed in \Cref{tbl:ppos}.
  Notice that every set of P-positions below is one of our sets $\P_i$ from \eqref{eq.p-i-def}: a finite subset of the P-positions of Wythoff's Game.

  \begin{table}[h!]
    \begin{center}
      \begin{tabular}{|c||l|}
        \hline
        $m$ & P-positions of $m$-Modular Wythoff's Game \\ \hline
        2 & $\P_1 = \{(0, 0), (1, 2), (2, 1)\}$ \\ \hline
        3 & $\P_1 = \{(0, 0), (1, 2), (2, 1)\}$ \\ \hline
        4 & $\P_2 = \{(0, 0), (1, 2), (2, 1), (3, 5), (5, 3)\}$ \\ \hline
        5 & $P_3 = \{(0, 0), (1, 2), (2, 1), (3, 5), (5, 3), (4, 7), (7, 4)\}$\\
        \hline
      \end{tabular}
    \end{center}
    \label{tbl:ppos}
  \end{table}

  While the pattern may not be obvious from the rows in this table alone, we will show in the next section that the P-positions of $m$-Modular Wythoff are precisely the P-positions in Wythoff's Game for which the smaller pile has strictly fewer than $m$ tokens.

  \section{P-positions of \texorpdfstring{$m$}{m}-Modular Wythoff}\label{sec:main}

  We make the following observation about $m$-Modular Wythoff before we consider its P-positions.

  \begin{lemma} \label{lem.move}
    Given positions $(q_1,q_2)\succ(s_1,s_2)$ with $q_1-q_2 \equiv s_1-s_2 \pmod{m}$, there is always a move from $(q_1,q_2)$ to $(s_1,s_2)$.
  \end{lemma}

  \begin{proof}
    We have three cases: $q_1=s_1$ and $q_2>s_2$; $q_1>s_1$ and $q_2=s_2$; or $q_1>s_1$ and $q_2>s_2$.
    In either of the first two cases, there is a Type I move from $(q_1,q_2)$ to $(s_1,s_2)$.
    In the third case, rearranging the equivalence from the hypothesis yields $q_1-s_1\equiv q_2-s_2\pmod{m}$, so removing $q_1-s_1>0$ tokens from the first pile and $q_2-s_2>0$ tokens from the second pile is a valid Type II move from $(q_1,q_2)$ to $(s_1,s_2)$.
  \end{proof}

  To describe the P-positions of $m$-Modular Wythoff, it will be helpful to introduce the following piece of notation.
  For any positive integer $m$, let $a_m$ be the number of lower Wythoff numbers strictly less than $m$.
  Equivalently, we define $a_m$ to be the unique integer for which $\lfloor a_m\phi \rfloor < m \leq \lfloor (a_m+1)\phi \rfloor$. In other words, $a_m = \lfloor m/\phi \rfloor$. Then the sequence $a_1, a_2, a_3, \ldots$ is sequence A005206 in \cite{OEIS} (with indices shifted by one).
  For visualization, here is a table of the sequence of $a_m$'s, where bold indices $m$ are lower Wythoff numbers.
  Each $a_m$ is then equal to the number of bold indices to its left.

  \begin{table}[h!]
    \begin{center}
      \begin{tabular}{|c||c|c|c|c|c|c|c|c|c|c|c|c|c|c|c|c|c|c|c|c|}
        \hline
        $m$ & \textbf{1} & 2 & \textbf{3} & \textbf{4} & 5 & \textbf{6} & 7 & \textbf{8} & \textbf{9} & 10 & \textbf{11} & \textbf{12} & 13 & \textbf{14} & 15 & \textbf{16} \\ \hline
        $a_m$ & 0 & 1 & 1 & 2 & 3 & 3 & 4 & 4 & 5 & 6 & 6 & 7 & 8 & 8 & 9 & 9
        \\ \hline
      \end{tabular}
    \end{center}
    \label{tbl:a-m}
  \end{table}

  We claim that the P-positions of $m$-Modular Wythoff form the set $\P_{a_m}$, which by construction contains $2a_m + 1 = 2\lfloor m/\phi \rfloor + 1$ elements. To prove our claim, we will make use of several simple observations regarding the set of positions $\P_{a_m}$.
  Throughout the rest of this section, we make use of the fact that $\phi$ is irrational and satisfies $\phi^2=\phi+1$ and therefore $1/\phi=\phi-1$.

  \begin{lemma} \label{lem.ppos-ineq}
    Let $(p_1,p_2)\in\P_{a_m}$ with $p_1<p_2$.
    Then
    \begin{align}
      p_1 &< m, \label{eq.p1-m} \\
      p_2 &< m\phi, \label{eq.p2-mphi} \\
      \text{and} \qquad p_2-p_1 &< m/\phi = m(\phi-1). \label{eq.p2p1-m-over-phi}
    \end{align}
  \end{lemma}
  \begin{proof}
    By our construction of $\P_{a_m}$ and our discussion in \Cref{sec:ppos-wyt}, we know that $p_1$ (resp.\ $p_2$) must be one of the first $a_m$ lower (resp.\ upper) Wythoff numbers, so in particular $p_1\leq\lf a_m\phi\rf$ and $p_2\leq\lf a_m\phi^2\rf$.
    But by construction $\lf a_m\phi\rf<m$ and thus $\lf a_m\phi^2\rf<m\phi$, so \eqref{eq.p1-m} and \eqref{eq.p2-mphi} follow.

    Moreover, again by our construction of $\P_{a_m}$ we have that $p_2-p_1\leq a_m$, and again by our construction of $a_m$ we have $a_m=\lf m/\phi\rf$.
    As $\phi$ is irrational and satisfies $1/\phi=\phi-1$, we have $\lf m/\phi\rf<m/\phi=m(\phi-1)$, so \eqref{eq.p2p1-m-over-phi} follows.
  \end{proof}

  \begin{lemma} \label{lem.coord-exists}
    For any nonnegative integer $r<m$, there is a position in $\P_{a_m}$ with $r$ as a coordinate.
  \end{lemma}
  \begin{proof}
    If $r=0$, then $(0,0)\in\P_{a_m}$ and we are done.
    So assume $r>0$.
    By \eqref{eq.ci} with $i\coloneqq a_m+1$, we have that $c_{a_m+1}$, the smallest positive integer not already a part of any ordered pair in $\P_{a_m}$, is equal to $\lf (a_m+1)\phi\rf$.
    But by the construction of $a_m$, we have that $m\leq\lf(a_m+1)\phi\rf=c_{a_m+1}$, so in particular $r<c_{a_m+1}$.
    Hence, by the minimality of $c_{a_m+1}$, there exists a position in $\P_{a_m}$ with $r$ as a coordinate.
  \end{proof}

  \begin{lemma} \label{lem.diff-exists}
    For any positive integer $s < m/\phi$, there are exactly two positions in $\P_{a_m}$ whose absolute difference in coordinates is $s$, namely $(q_1,q_2)$ and $(q_2,q_1)$ in $\P_{a_m}$ with $q_1 = \lf s\phi\rf$ and $q_2 = \lf s\phi^2\rf$.
  \end{lemma}
  \begin{proof}
    As $\phi$ is irrational, $m/\phi$ is not an integer, but $s$ is; so $s\leq\lf m/\phi\rf$.
    Since $\lf m/\phi\rf=a_m$ by construction, it follows that $s\leq a_m$.
    So by \eqref{eq.p-i-def},
    \[
    (c_s,c_s+s),(c_s+s,c_s)\in\P_s\s\P_{a_m},
    \]
    with
    \[
    c_s=\lf s\phi\rf \qquad\text{and}\qquad c_s+s=\lf s\phi^2\rf
    \]
    by \eqref{eq.ci}.
    Hence $(c_s,c_s+s)$ and $(c_s+s,c_s)$ are the positions in $\P_{a_m}$ we seek; by the construction of $\P_{a_m}$, every other pair in $\P_{a_m}$ has a different absolute difference of coordinates.
  \end{proof}

  Equipped with these lemmas, we are now ready to describe the set of P-positions in $m$-Modular Wythoff.

  \begin{theorem}
    The P-positions of $m$-Modular Wythoff form the set $\P_{a_m}$.
  \end{theorem}

  \begin{proof}
    We know that the terminal position $(0,0)\in\P_{a_m}$, so it suffices to show that \textbf{there are no moves between any two positions in $\P_{a_m}$} and that \textbf{from any position not in $\P_{a_m}$, there exists a move to a position in $\P_{a_m}$}.

    \textbf{First, we show that there are no moves between any two positions in $\P_{a_m}$.}
    As $\P_{a_m}$ is a subset of the set $\P$ of P-positions in Wythoff's Game, there can be no Wythoff moves between its elements: there can be no Type I moves between them, and there can be no Type II moves between them in which the same number of tokens is removed from either pile.
    It remains to verify that no other Type II moves are possible between elements of $\P_{a_m}$.

    Let $(p_1,p_2)$ be a non-terminal position in $\P_{a_m}$ with $p_1<p_2$ (the case of $p_1>p_2$ is analogous by symmetry), and assume to the contrary that there exists a Type II move from $(p_1,p_2)$ to another position in $\P_{a_m}$ that removes a different number of tokens from each pile.
    That is, there exist positive integers $k_1,k_2$ with $k_1\equiv k_2\pmod{m}$ but $k_1\neq k_2$ such that $(p_1-k_1,p_2-k_2)\in\P_{a_m}$ as well.
    In particular, both piles must have a nonnegative number of tokens, so $k_1\leq p_1$ and $k_2\leq p_2$.
    Then by \eqref{eq.p1-m}, $p_1<m$, so $k_1<m$; and by \eqref{eq.p2-mphi}, $p_2<m\phi$, so $k_2<m\phi<2m$.
    Hence the only way for $k_2-k_1$ to be a nonzero multiple of $m$ would be for $k_2=k_1+m$.

    So the Type II move in question must lead to $(p_1-k_1,p_2-k_1-m)\in\P_{a_m}$.
    By \eqref{eq.p2p1-m-over-phi}, we have $p_2-p_1<m/\phi=m(\phi-1)$, which in turn is less than $m$; hence
    \[
    (p_1-k_1)-(p_2-k_1-m)=m-(p_2-p_1)>0,
    \]
    so $p_1-k_1$ is the larger of the two piles and $m-(p_2-p_1)$ is the absolute difference between them.
    Then by \Cref{lem.diff-exists}, we have that $p_1-k_1=\lf(m-(p_2-p_1))\phi^2\rf$, so $(m-(p_2-p_1))\phi^2 < p_1$.
    Since
    \[
    1/\phi^2=(\phi-1)^2=\phi^2-2\phi+1=\phi+1-2\phi+1=2-\phi,
    \]
    it follows that
    \[
    m-(p_2-p_1) < p_1/\phi^2 = p_1(2-\phi) < m(2-\phi).
    \]
    Yet we still have $p_2-p_1 < m(\phi-1)$.
    Adding these inequalities yields $m<m$, a contradiction.
    Therefore no moves are possible between elements of $\P_{a_m}$.

    \textbf{Next, we show that from any position not in $\P_{a_m}$, there exists a move to a position in $\P_{a_m}$.}
    Let $(q_1, q_2)$ be a position not in $\P_{a_m}$ with $q_1\leq q_2$ (again, the case of $q_1\geq q_2$ is analogous by symmetry).
    We wish to show that there exists a move from $(q_1,q_2)$ to a position in $\P_{a_m}$.
    There are two cases to consider: either $q_1 < m$ or $q_1 \geq m$.

    \begin{enumerate}[label=\textbf{Case \arabic*:}]
      \item Say $q_1 < m$.
      If $q_1=0$, then for $(q_1, q_2)\notin\P_{a_m}$ to hold we must have $q_2\neq0$.
      Hence there is a Type I move from $(q_1,q_2)=(0,q_2)$ to $(0,0)\in\P_{a_m}$, and we are done.

      So assume instead $q_1>0$.
      Then by \Cref{lem.coord-exists}, there exists a position in $\P_{a_m}$ with $q_1$ as the size of one of the piles: either the larger pile or the smaller pile.
      If $q_1$ is the larger pile size of a position in $\P_{a_m}$, then there exists $q_2'<q_1\leq q_2$ with $(q_1,q_2')\in\P_{a_m}$.
      Hence there is a Type I move from $(q_1, q_2)$ to $(q_1,q_2')\in\P_{a_m}$, and we are again done.

      Otherwise, $q_1$ is the smaller pile size of a position in $\P_{a_m}$, so it must be a lower Wythoff number: $q_1 = \lf i\phi\rf$ for some positive integer $i\leq a_m$, with $(q_1,q_1+i)\in\P_{a_m}$.
      We have two subcases: either $q_2>q_1\phi$ or $q_2\leq q_1\phi$.

      If $q_2 > q_1\phi = \lf i\phi\rf\phi$, then since $\phi=1+1/\phi$, we have
      \[
      q_2 > \lf i\phi\rf\phi = \lf i\phi\rf + \lf i\phi\rf/\phi > \lf i\phi\rf+i-1,
      \]
      where the latter inequality follows from the fact that $i\phi-(i-1)\phi=\phi>1$ and thus $\lf i\phi\rf > (i-1)\phi$.
      It follows that $q_2\geq\lf i\phi\rf+i=q_1+i$.
      Yet $(q_1,q_1+i)$ is in $\P_{a_m}$ while $(q_1,q_2)$ is not; so in fact $q_2>q_1+i$, and there is a Type I move from $(q_1,q_2)$ to $(q_1,q_1+i)\in\P_{a_m}$, as desired.

      On the other hand, if $q_2 \leq q_1 \phi$, then
      \[
      q_2 - q_1 \leq q_1\phi - q_1 = q_1/\phi = \lfloor i\phi \rfloor/\phi < i \leq a_m.
      \]
      So either $q_2-q_1=0$, in which case there is a Type II move from $(q_1,q_2)$ to $(0,0)\in\P_{a_m}$, and we are done; or $q_2-q_1>0$, in which case \Cref{lem.diff-exists} implies that $(\lf(q_2 - q_1)\phi\rf, \lf(q_2 - q_1)\phi\rf+q_2-q_1)\in\P_{a_m}$.
      Then from above we have $i>q_2-q_1$, so $q_1=\lf i\phi\rf>\lf(q_2-q_1)\phi\rf$, and therefore we have a Type II move from $(q_1,q_2)$ to $(\lf(q_2 - q_1)\phi\rf, \lf(q_2 - q_1)\phi\rf+q_2-q_1)\in\P_{a_m}$ given by removing the same number of tokens from each pile, as desired.

      \item Say instead $q_1 \geq m$.
      By the division algorithm, there exist nonnegative integers $r$ and $x$ satisfying $q_2-q_1=mx+r$ and $0\leq r<m$.
      We have three subcases: it could be that $r=0$; it could be that $0<r<m/\phi$; or it could be that $m/\phi<r<m$.

      If $r=0$, then $m$ divides $q_2-q_1$, so removing $q_1$ from one pile and $q_2$ from the other is a Type II move sending $(q_1,q_2)$ to $(0,0)\in\P_{a_m}$, as desired.

      If on the other hand $0<r<m/\phi$, then by \Cref{lem.diff-exists} there exists a position $(s_1,s_2)\in\P_{a_m}$ for which $s_2-s_1=r$, namely when $s_1=\lf r\phi\rf$.
      We have
      \[
      s_1=\lf r\phi\rf<\lf (m/\phi)\phi\rf=m\leq q_1
      \]
      and thus
      \[
      s_2=s_1+r<q_1+r=q_2-mx\leq q_2,
      \]
      so $(q_1,q_2)\succ(s_1,s_2)$.
      As $s_2-s_1=r\equiv q_2-q_1\pmod{m}$, it follows from \Cref{lem.move} that there is a move from $(q_1,q_2)$ to $(s_1,s_2)\in\P_{a_m}$, as desired.

      Otherwise, $m/\phi<r<m$.
      Since $q_2-q_1\equiv r\pmod{m}$, we have that $q_1-q_2\equiv m-r\pmod{m}$.
      Note that $\phi<2$, so $1<2/\phi$ and thus $1-1/\phi<1/\phi$.
      It follows that
      \[
      0<m-r<m-m/\phi<m/\phi,
      \]
      so by \Cref{lem.diff-exists}, there exists a position $(s_2,s_1)\in\P_{a_m}$ for which $s_2-s_1=m-r$ and thus $q_1 - q_2 \equiv s_2 - s_1 \pmod{m}$; in particular, $s_2=\lf(m-r)\phi^2\rf$.
      Then
      \[
      s_1<s_2 = \lf(m-r)\phi^2\rf < \lf(m-m/\phi)\phi^2\rf = \lf m(\phi^2-\phi)\rf = m \leq q_1 \leq q_2,
      \]
      so $(q_1, q_2)\succ(s_2, s_1)$.
      Hence, by \Cref{lem.move}, there is a move from $(q_1, q_2)$ to $(s_1, s_2)$, and we are done. \qedhere
    \end{enumerate}
  \end{proof}

  \begin{corollary}
    The number of P-positions of $m$-Modular Wythoff is finite; in particular, it is equal to $2\lf m/\phi\rf + 1$.
  \end{corollary}
  \begin{proof}
    By construction, $\P_{a_m}$ contains $2a_m + 1 = 2\lf m/\phi\rf + 1$ elements.
  \end{proof}

  \section{Further directions} \label{sec:fd}

  We conclude by suggesting some further directions of study.

  \subsection{Grundy numbers for \texorpdfstring{$m$}{m}-Modular Wythoff}

  The \textit{Grundy number} of a position in a combinatorial game indicates the size of the pile from the Game of Nim to which the position is equivalent.
  A position's Grundy number can be recursively computed as the minimal excluded nonnegative integer among the Grundy numbers of the positions to which the original position can go in a single move.
  In particular, every P-position has a Grundy number of $0$, while the N-positions have Grundy numbers that are positive.

  As of \cite{wyt-grun}, no closed-form formula for the Grundy numbers for Wythoff's Game are known.
  Nevertheless, we list a few recursively computed Grundy numbers for Wythoff's Game in \Cref{tbl:wyt-grun}; there, as in the tables that follow, the entry in the row labeled $p_1$ and the column labeled $p_2$ indicates the Grundy number of the position $(p_1,p_2)$.

  The other tables in the appendix list some of the Grundy numbers for our $m$-Modular Wythoff in the cases of $2\leq m\leq9$.
  We make a few initial observations:
  \begin{itemize}
    \item Grundy numbers for $m$-Modular Wythoff coincide with the Grundy numbers for the usual Wythoff's Game when the pile sizes are at most $m$, as in those cases the moves available for the two games are identical.
    \item There \textit{is} a clear pattern in the Grundy numbers for 3-Modular Wythoff: the Grundy number of position $(3i+r,3j+s)$, where $i,j,r,s$ are nonnegative integers with $r,s\in\{0,1,2\}$, is equal to $3(i+j)+t$ for $t\in\{0,1,2\}$ satisfying $r+s\equiv t\pmod{3}$.
    This can be proven by induction on $i$ and $j$.
    The fact that the case of $m=3$, along with of course the trivial case of $m=1$, has such a concise characterization of its Grundy numbers is likely to do with the fact that $m$-Modular Wythoff has exactly $m$ P-positions precisely when $m=3$ or $m=1$.
  \end{itemize}
  But ultimately, just like with Wythoff's Game, we do not have a closed-form formula for the Grundy numbers of $m$-Modular Wythoff for general values of $m$ (aside from $1$ and $3$).
  Nor do we know precisely how the Grundy numbers for $m$-Modular Wythoff relate to the Grundy numbers for Wythoff's Game for larger pile sizes.
  Further investigation is needed to compute Grundy numbers for $m$-Modular Wythoff, to see if known results for the Grundy numbers of Wythoff's Game such as those given in \cite{wyt-grun} have analogs in $m$-Modular variants, and perhaps eventually use the properties of the Grundy numbers for $m$-Modular Wythoff to uncover further properties of the Grundy numbers for standard Wythoff.

  \subsection{Extensions with finite P-positions}

  Wythoff's Game has an infinite set of P-positions; and yet we showed that by expanding the possible moves of the game in a certain way, we form an extension of Wythoff's Game, namely $m$-Modular Wythoff, that has a finite set of P-positions---indeed, a subset of the P-positions of the original Wythoff's Game.
  This raises the following questions, among others:
  \begin{itemize}
    \item Can we characterize the additional moves necessary and/or sufficient to extend Wythoff's Game so that the extension has a finite set of P-positions?
    \item Can we characterize the additional moves necessary and/or sufficient to extend Wythoff's Game so that the extension has a finite subset of the P-positions of the original Wythoff's Game?
    \item Can we characterize the additional moves necessary and/or sufficient to extend another invariant game so that the extension has a finite set of P-positions, particularly a finite subset of the P-positions of the original game?
  \end{itemize}

  The latter two questions, generalizing to other invariant combinatorial games beyond Wythoff's Game, lead to the final avenue of further study we shall mention.

  \subsection{Modular variants of other games}

  While a modular extension of Nim was studied in \cite{KS} and a modular extension of Wythoff was studied here, modular extensions of other combinatorial games have not been studied in any level of generalization.
  We are interested in whether modular variants of other games relate to the original games much like how the modular variants of Nim or Wythoff relate to standard Nim or Wythoff.

  \printbibliography

  \section*{Appendix: Tables of Grundy numbers}

  In the following tables, we list the Grundy numbers for Wythoff's Game, as well as $m$-Modular Wythoff's Game for $2\leq m\leq9$, for piles of size at most $15$.

  \begin{table}[h!]
    \scriptsize
    \caption{Grundy numbers for Wythoff's Game.}
    \label{tbl:wyt-grun}
    \begin{center}
      \begin{tabular}{c|cccccccccccccccc}
        & 0 & 1 & 2 & 3 & 4 & 5 & 6 & 7 & 8 & 9 & 10 & 11 & 12 & 13 & 14 & 15 \\
        \hline
        0 & 0 & 1 & 2 & 3 & 4 & 5 & 6 & 7 & 8 & 9 & 10 & 11 & 12 & 13 & 14 & 15 \\
        1 & 1 & 2 & 0 & 4 & 5 & 3 & 7 & 8 & 6 & 10 & 11 & 9 & 13 & 14 & 12 & 16 \\
        2 & 2 & 0 & 1 & 5 & 3 & 4 & 8 & 6 & 7 & 11 & 9 & 10 & 14 & 12 & 13 & 17 \\
        3 & 3 & 4 & 5 & 6 & 2 & 0 & 1 & 9 & 10 & 12 & 8 & 7 & 15 & 11 & 16 & 18 \\
        4 & 4 & 5 & 3 & 2 & 7 & 6 & 9 & 0 & 1 & 8 & 13 & 12 & 11 & 16 & 15 & 10 \\
        5 & 5 & 3 & 4 & 0 & 6 & 8 & 10 & 1 & 2 & 7 & 12 & 14 & 9 & 15 & 17 & 13 \\
        6 & 6 & 7 & 8 & 1 & 9 & 10 & 3 & 4 & 5 & 13 & 0 & 2 & 16 & 17 & 18 & 12 \\
        7 & 7 & 8 & 6 & 9 & 0 & 1 & 4 & 5 & 3 & 14 & 15 & 13 & 17 & 2 & 10 & 19 \\
        8 & 8 & 6 & 7 & 10 & 1 & 2 & 5 & 3 & 4 & 15 & 16 & 17 & 18 & 0 & 9 & 14 \\
        9 & 9 & 10 & 11 & 12 & 8 & 7 & 13 & 14 & 15 & 16 & 17 & 6 & 19 & 5 & 1 & 0 \\
        10 & 10 & 11 & 9 & 8 & 13 & 12 & 0 & 15 & 16 & 17 & 14 & 18 & 7 & 6 & 2 & 3 \\
        11 & 11 & 9 & 10 & 7 & 12 & 14 & 2 & 13 & 17 & 6 & 18 & 15 & 8 & 19 & 20 & 21 \\
        12 & 12 & 13 & 14 & 15 & 11 & 9 & 16 & 17 & 18 & 19 & 7 & 8 & 10 & 20 & 21 & 22 \\
        13 & 13 & 14 & 12 & 11 & 16 & 15 & 17 & 2 & 0 & 5 & 6 & 19 & 20 & 9 & 7 & 8 \\
        14 & 14 & 12 & 13 & 16 & 15 & 17 & 18 & 10 & 9 & 1 & 2 & 20 & 21 & 7 & 11 & 23 \\
        15 & 15 & 16 & 17 & 18 & 10 & 13 & 12 & 19 & 14 & 0 & 3 & 21 & 22 & 8 & 23 & 20
      \end{tabular}
    \end{center}
  \end{table}

  \begin{table}[h!]
    \scriptsize
    \caption{Grundy numbers for 2-Modular Wythoff's Game.}
    \begin{center}
      \begin{tabular}{c|cccccccccccccccc}
        & 0 & 1 & 2 & 3 & 4 & 5 & 6 & 7 & 8 & 9 & 10 & 11 & 12 & 13 & 14 & 15 \\
        \hline
        0 & 0 & 1 & 2 & 3 & 4 & 5 & 6 & 7 & 8 & 9 & 10 & 11 & 12 & 13 & 14 & 15 \\
        1 & 1 & 2 & 0 & 4 & 5 & 3 & 7 & 8 & 6 & 10 & 11 & 9 & 13 & 14 & 12 & 16 \\
        2 & 2 & 0 & 1 & 5 & 3 & 4 & 8 & 6 & 7 & 11 & 9 & 10 & 14 & 12 & 13 & 17 \\
        3 & 3 & 4 & 5 & 6 & 2 & 7 & 9 & 10 & 11 & 12 & 8 & 13 & 15 & 16 & 17 & 18 \\
        4 & 4 & 5 & 3 & 2 & 7 & 6 & 10 & 11 & 9 & 8 & 13 & 12 & 16 & 17 & 15 & 14 \\
        5 & 5 & 3 & 4 & 7 & 6 & 8 & 11 & 9 & 10 & 13 & 12 & 14 & 17 & 15 & 16 & 19 \\
        6 & 6 & 7 & 8 & 9 & 10 & 11 & 5 & 12 & 13 & 14 & 15 & 16 & 18 & 19 & 20 & 21 \\
        7 & 7 & 8 & 6 & 10 & 11 & 9 & 12 & 13 & 14 & 15 & 16 & 17 & 19 & 20 & 18 & 22 \\
        8 & 8 & 6 & 7 & 11 & 9 & 10 & 13 & 14 & 12 & 16 & 17 & 15 & 20 & 18 & 19 & 23 \\
        9 & 9 & 10 & 11 & 12 & 8 & 13 & 14 & 15 & 16 & 17 & 18 & 19 & 21 & 22 & 23 & 24 \\
        10 & 10 & 11 & 9 & 8 & 13 & 12 & 15 & 16 & 17 & 18 & 14 & 20 & 22 & 23 & 21 & 25 \\
        11 & 11 & 9 & 10 & 13 & 12 & 14 & 16 & 17 & 15 & 19 & 20 & 18 & 23 & 21 & 22 & 26 \\
        12 & 12 & 13 & 14 & 15 & 16 & 17 & 18 & 19 & 20 & 21 & 22 & 23 & 11 & 24 & 25 & 27 \\
        13 & 13 & 14 & 12 & 16 & 17 & 15 & 19 & 20 & 18 & 22 & 23 & 21 & 24 & 25 & 26 & 28 \\
        14 & 14 & 12 & 13 & 17 & 15 & 16 & 20 & 18 & 19 & 23 & 21 & 22 & 25 & 26 & 24 & 29 \\
        15 & 15 & 16 & 17 & 18 & 14 & 19 & 21 & 22 & 23 & 24 & 25 & 26 & 27 & 28 & 29 & 30
      \end{tabular}
    \end{center}
  \end{table}

  \begin{table}[h!]
    \scriptsize
    \caption{Grundy numbers for 3-Modular Wythoff's Game.}
    \begin{center}
      \begin{tabular}{c|cccccccccccccccc}
        & 0 & 1 & 2 & 3 & 4 & 5 & 6 & 7 & 8 & 9 & 10 & 11 & 12 & 13 & 14 & 15 \\
        \hline
        0 & 0 & 1 & 2 & 3 & 4 & 5 & 6 & 7 & 8 & 9 & 10 & 11 & 12 & 13 & 14 & 15 \\
        1 & 1 & 2 & 0 & 4 & 5 & 3 & 7 & 8 & 6 & 10 & 11 & 9 & 13 & 14 & 12 & 16 \\
        2 & 2 & 0 & 1 & 5 & 3 & 4 & 8 & 6 & 7 & 11 & 9 & 10 & 14 & 12 & 13 & 17 \\
        3 & 3 & 4 & 5 & 6 & 7 & 8 & 9 & 10 & 11 & 12 & 13 & 14 & 15 & 16 & 17 & 18 \\
        4 & 4 & 5 & 3 & 7 & 8 & 6 & 10 & 11 & 9 & 13 & 14 & 12 & 16 & 17 & 15 & 19 \\
        5 & 5 & 3 & 4 & 8 & 6 & 7 & 11 & 9 & 10 & 14 & 12 & 13 & 17 & 15 & 16 & 20 \\
        6 & 6 & 7 & 8 & 9 & 10 & 11 & 12 & 13 & 14 & 15 & 16 & 17 & 18 & 19 & 20 & 21 \\
        7 & 7 & 8 & 6 & 10 & 11 & 9 & 13 & 14 & 12 & 16 & 17 & 15 & 19 & 20 & 18 & 22 \\
        8 & 8 & 6 & 7 & 11 & 9 & 10 & 14 & 12 & 13 & 17 & 15 & 16 & 20 & 18 & 19 & 23 \\
        9 & 9 & 10 & 11 & 12 & 13 & 14 & 15 & 16 & 17 & 18 & 19 & 20 & 21 & 22 & 23 & 24 \\
        10 & 10 & 11 & 9 & 13 & 14 & 12 & 16 & 17 & 15 & 19 & 20 & 18 & 22 & 23 & 21 & 25 \\
        11 & 11 & 9 & 10 & 14 & 12 & 13 & 17 & 15 & 16 & 20 & 18 & 19 & 23 & 21 & 22 & 26 \\
        12 & 12 & 13 & 14 & 15 & 16 & 17 & 18 & 19 & 20 & 21 & 22 & 23 & 24 & 25 & 26 & 27 \\
        13 & 13 & 14 & 12 & 16 & 17 & 15 & 19 & 20 & 18 & 22 & 23 & 21 & 25 & 26 & 24 & 28 \\
        14 & 14 & 12 & 13 & 17 & 15 & 16 & 20 & 18 & 19 & 23 & 21 & 22 & 26 & 24 & 25 & 29 \\
        15 & 15 & 16 & 17 & 18 & 19 & 20 & 21 & 22 & 23 & 24 & 25 & 26 & 27 & 28 & 29 & 30
      \end{tabular}
    \end{center}
  \end{table}

  \begin{table}[h!]
    \scriptsize
    \caption{Grundy numbers for 4-Modular Wythoff's Game.}
    \begin{center}
      \begin{tabular}{c|cccccccccccccccc}
        & 0 & 1 & 2 & 3 & 4 & 5 & 6 & 7 & 8 & 9 & 10 & 11 & 12 & 13 & 14 & 15 \\
        \hline
        0 & 0 & 1 & 2 & 3 & 4 & 5 & 6 & 7 & 8 & 9 & 10 & 11 & 12 & 13 & 14 & 15 \\
        1 & 1 & 2 & 0 & 4 & 5 & 3 & 7 & 8 & 6 & 10 & 11 & 9 & 13 & 14 & 12 & 16 \\
        2 & 2 & 0 & 1 & 5 & 3 & 4 & 8 & 6 & 7 & 11 & 9 & 10 & 14 & 12 & 13 & 17 \\
        3 & 3 & 4 & 5 & 6 & 2 & 0 & 9 & 10 & 11 & 1 & 8 & 7 & 15 & 16 & 17 & 18 \\
        4 & 4 & 5 & 3 & 2 & 7 & 6 & 1 & 11 & 9 & 8 & 12 & 13 & 16 & 17 & 15 & 10 \\
        5 & 5 & 3 & 4 & 0 & 6 & 8 & 10 & 9 & 12 & 13 & 14 & 15 & 17 & 11 & 16 & 19 \\
        6 & 6 & 7 & 8 & 9 & 1 & 10 & 5 & 12 & 13 & 14 & 15 & 16 & 11 & 18 & 19 & 20 \\
        7 & 7 & 8 & 6 & 10 & 11 & 9 & 12 & 13 & 14 & 5 & 16 & 17 & 18 & 19 & 20 & 21 \\
        8 & 8 & 6 & 7 & 11 & 9 & 12 & 13 & 14 & 15 & 16 & 17 & 18 & 19 & 20 & 21 & 22 \\
        9 & 9 & 10 & 11 & 1 & 8 & 13 & 14 & 5 & 16 & 12 & 18 & 19 & 20 & 21 & 22 & 23 \\
        10 & 10 & 11 & 9 & 8 & 12 & 14 & 15 & 16 & 17 & 18 & 19 & 20 & 21 & 22 & 23 & 24 \\
        11 & 11 & 9 & 10 & 7 & 13 & 15 & 16 & 17 & 18 & 19 & 20 & 14 & 22 & 23 & 24 & 25 \\
        12 & 12 & 13 & 14 & 15 & 16 & 17 & 11 & 18 & 19 & 20 & 21 & 22 & 23 & 24 & 25 & 26 \\
        13 & 13 & 14 & 12 & 16 & 17 & 11 & 18 & 19 & 20 & 21 & 22 & 23 & 24 & 25 & 26 & 27 \\
        14 & 14 & 12 & 13 & 17 & 15 & 16 & 19 & 20 & 21 & 22 & 23 & 24 & 25 & 26 & 18 & 28 \\
        15 & 15 & 16 & 17 & 18 & 10 & 19 & 20 & 21 & 22 & 23 & 24 & 25 & 26 & 27 & 28 & 29
      \end{tabular}
    \end{center}
  \end{table}

  \begin{table}[h!]
    \scriptsize
    \caption{Grundy numbers for 5-Modular Wythoff's Game.}
    \begin{center}
      \begin{tabular}{c|cccccccccccccccc}
        & 0 & 1 & 2 & 3 & 4 & 5 & 6 & 7 & 8 & 9 & 10 & 11 & 12 & 13 & 14 & 15 \\
        \hline
        0 & 0 & 1 & 2 & 3 & 4 & 5 & 6 & 7 & 8 & 9 & 10 & 11 & 12 & 13 & 14 & 15 \\
        1 & 1 & 2 & 0 & 4 & 5 & 3 & 7 & 8 & 6 & 10 & 11 & 9 & 13 & 14 & 12 & 16 \\
        2 & 2 & 0 & 1 & 5 & 3 & 4 & 8 & 6 & 7 & 11 & 9 & 10 & 14 & 12 & 13 & 17 \\
        3 & 3 & 4 & 5 & 6 & 2 & 0 & 1 & 9 & 10 & 12 & 8 & 7 & 15 & 11 & 16 & 18 \\
        4 & 4 & 5 & 3 & 2 & 7 & 6 & 9 & 0 & 11 & 8 & 13 & 1 & 16 & 17 & 15 & 10 \\
        5 & 5 & 3 & 4 & 0 & 6 & 8 & 10 & 1 & 9 & 7 & 12 & 14 & 17 & 15 & 18 & 13 \\
        6 & 6 & 7 & 8 & 1 & 9 & 10 & 3 & 11 & 12 & 13 & 14 & 4 & 18 & 16 & 17 & 19 \\
        7 & 7 & 8 & 6 & 9 & 0 & 1 & 11 & 4 & 13 & 14 & 15 & 12 & 19 & 20 & 10 & 21 \\
        8 & 8 & 6 & 7 & 10 & 11 & 9 & 12 & 13 & 14 & 15 & 16 & 17 & 20 & 18 & 19 & 22 \\
        9 & 9 & 10 & 11 & 12 & 8 & 7 & 13 & 14 & 15 & 16 & 17 & 18 & 21 & 19 & 20 & 23 \\
        10 & 10 & 11 & 9 & 8 & 13 & 12 & 14 & 15 & 16 & 17 & 18 & 19 & 22 & 23 & 21 & 24 \\
        11 & 11 & 9 & 10 & 7 & 1 & 14 & 4 & 12 & 17 & 18 & 19 & 13 & 23 & 21 & 22 & 25 \\
        12 & 12 & 13 & 14 & 15 & 16 & 17 & 18 & 19 & 20 & 21 & 22 & 23 & 11 & 24 & 25 & 26 \\
        13 & 13 & 14 & 12 & 11 & 17 & 15 & 16 & 20 & 18 & 19 & 23 & 21 & 24 & 22 & 26 & 27 \\
        14 & 14 & 12 & 13 & 16 & 15 & 18 & 17 & 10 & 19 & 20 & 21 & 22 & 25 & 26 & 23 & 28 \\
        15 & 15 & 16 & 17 & 18 & 10 & 13 & 19 & 21 & 22 & 23 & 24 & 25 & 26 & 27 & 28 & 29
      \end{tabular}
    \end{center}
  \end{table}

  \begin{table}[h!]
    \scriptsize
    \caption{Grundy numbers for 6-Modular Wythoff's Game.}
    \begin{center}
      \begin{tabular}{c|cccccccccccccccc}
        & 0 & 1 & 2 & 3 & 4 & 5 & 6 & 7 & 8 & 9 & 10 & 11 & 12 & 13 & 14 & 15 \\
        \hline
        0 & 0 & 1 & 2 & 3 & 4 & 5 & 6 & 7 & 8 & 9 & 10 & 11 & 12 & 13 & 14 & 15 \\
        1 & 1 & 2 & 0 & 4 & 5 & 3 & 7 & 8 & 6 & 10 & 11 & 9 & 13 & 14 & 12 & 16 \\
        2 & 2 & 0 & 1 & 5 & 3 & 4 & 8 & 6 & 7 & 11 & 9 & 10 & 14 & 12 & 13 & 17 \\
        3 & 3 & 4 & 5 & 6 & 2 & 0 & 1 & 9 & 10 & 12 & 8 & 7 & 15 & 11 & 16 & 18 \\
        4 & 4 & 5 & 3 & 2 & 7 & 6 & 9 & 0 & 1 & 8 & 13 & 12 & 11 & 16 & 15 & 14 \\
        5 & 5 & 3 & 4 & 0 & 6 & 8 & 10 & 1 & 2 & 7 & 12 & 14 & 9 & 15 & 17 & 13 \\
        6 & 6 & 7 & 8 & 1 & 9 & 10 & 3 & 4 & 5 & 13 & 14 & 15 & 16 & 17 & 18 & 12 \\
        7 & 7 & 8 & 6 & 9 & 0 & 1 & 4 & 5 & 3 & 14 & 15 & 13 & 17 & 10 & 19 & 20 \\
        8 & 8 & 6 & 7 & 10 & 1 & 2 & 5 & 3 & 4 & 15 & 16 & 17 & 18 & 9 & 11 & 21 \\
        9 & 9 & 10 & 11 & 12 & 8 & 7 & 13 & 14 & 15 & 16 & 17 & 6 & 19 & 20 & 21 & 22 \\
        10 & 10 & 11 & 9 & 8 & 13 & 12 & 14 & 15 & 16 & 17 & 18 & 19 & 20 & 6 & 22 & 23 \\
        11 & 11 & 9 & 10 & 7 & 12 & 14 & 15 & 13 & 17 & 6 & 19 & 20 & 21 & 18 & 8 & 24 \\
        12 & 12 & 13 & 14 & 15 & 11 & 9 & 16 & 17 & 18 & 19 & 20 & 21 & 10 & 22 & 23 & 7 \\
        13 & 13 & 14 & 12 & 11 & 16 & 15 & 17 & 10 & 9 & 20 & 6 & 18 & 22 & 19 & 24 & 25 \\
        14 & 14 & 12 & 13 & 16 & 15 & 17 & 18 & 19 & 11 & 21 & 22 & 8 & 23 & 24 & 9 & 26 \\
        15 & 15 & 16 & 17 & 18 & 14 & 13 & 12 & 20 & 21 & 22 & 23 & 24 & 7 & 25 & 26 & 27
      \end{tabular}
    \end{center}
  \end{table}

  \begin{table}[h!]
    \scriptsize
    \caption{Grundy numbers for 7-Modular Wythoff's Game.}
    \begin{center}
      \begin{tabular}{c|cccccccccccccccc}
        & 0 & 1 & 2 & 3 & 4 & 5 & 6 & 7 & 8 & 9 & 10 & 11 & 12 & 13 & 14 & 15 \\
        \hline
        0 & 0 & 1 & 2 & 3 & 4 & 5 & 6 & 7 & 8 & 9 & 10 & 11 & 12 & 13 & 14 & 15 \\
        1 & 1 & 2 & 0 & 4 & 5 & 3 & 7 & 8 & 6 & 10 & 11 & 9 & 13 & 14 & 12 & 16 \\
        2 & 2 & 0 & 1 & 5 & 3 & 4 & 8 & 6 & 7 & 11 & 9 & 10 & 14 & 12 & 13 & 17 \\
        3 & 3 & 4 & 5 & 6 & 2 & 0 & 1 & 9 & 10 & 12 & 8 & 7 & 15 & 11 & 16 & 18 \\
        4 & 4 & 5 & 3 & 2 & 7 & 6 & 9 & 0 & 1 & 8 & 13 & 12 & 11 & 16 & 15 & 10 \\
        5 & 5 & 3 & 4 & 0 & 6 & 8 & 10 & 1 & 2 & 7 & 12 & 14 & 9 & 15 & 17 & 13 \\
        6 & 6 & 7 & 8 & 1 & 9 & 10 & 3 & 4 & 11 & 13 & 0 & 15 & 16 & 5 & 18 & 12 \\
        7 & 7 & 8 & 6 & 9 & 0 & 1 & 4 & 5 & 3 & 14 & 15 & 2 & 17 & 18 & 10 & 19 \\
        8 & 8 & 6 & 7 & 10 & 1 & 2 & 11 & 3 & 4 & 15 & 16 & 17 & 18 & 19 & 9 & 20 \\
        9 & 9 & 10 & 11 & 12 & 8 & 7 & 13 & 14 & 15 & 16 & 17 & 18 & 6 & 20 & 21 & 22 \\
        10 & 10 & 11 & 9 & 8 & 13 & 12 & 0 & 15 & 16 & 17 & 14 & 19 & 20 & 21 & 6 & 23 \\
        11 & 11 & 9 & 10 & 7 & 12 & 14 & 15 & 2 & 17 & 18 & 19 & 13 & 21 & 22 & 20 & 24 \\
        12 & 12 & 13 & 14 & 15 & 11 & 9 & 16 & 17 & 18 & 6 & 20 & 21 & 10 & 23 & 19 & 25 \\
        13 & 13 & 14 & 12 & 11 & 16 & 15 & 5 & 18 & 19 & 20 & 21 & 22 & 23 & 17 & 24 & 26 \\
        14 & 14 & 12 & 13 & 16 & 15 & 17 & 18 & 10 & 9 & 21 & 6 & 20 & 19 & 24 & 22 & 27 \\
        15 & 15 & 16 & 17 & 18 & 10 & 13 & 12 & 19 & 20 & 22 & 23 & 24 & 25 & 26 & 27 & 21
      \end{tabular}
    \end{center}
  \end{table}

  \begin{table}[h!]
    \scriptsize
    \caption{Grundy numbers for 8-Modular Wythoff's Game.}
    \begin{center}
      \begin{tabular}{c|cccccccccccccccc}
        & 0 & 1 & 2 & 3 & 4 & 5 & 6 & 7 & 8 & 9 & 10 & 11 & 12 & 13 & 14 & 15 \\
        \hline
        0 & 0 & 1 & 2 & 3 & 4 & 5 & 6 & 7 & 8 & 9 & 10 & 11 & 12 & 13 & 14 & 15 \\
        1 & 1 & 2 & 0 & 4 & 5 & 3 & 7 & 8 & 6 & 10 & 11 & 9 & 13 & 14 & 12 & 16 \\
        2 & 2 & 0 & 1 & 5 & 3 & 4 & 8 & 6 & 7 & 11 & 9 & 10 & 14 & 12 & 13 & 17 \\
        3 & 3 & 4 & 5 & 6 & 2 & 0 & 1 & 9 & 10 & 12 & 8 & 7 & 15 & 11 & 16 & 18 \\
        4 & 4 & 5 & 3 & 2 & 7 & 6 & 9 & 0 & 1 & 8 & 13 & 12 & 11 & 16 & 15 & 10 \\
        5 & 5 & 3 & 4 & 0 & 6 & 8 & 10 & 1 & 2 & 7 & 12 & 14 & 9 & 15 & 17 & 13 \\
        6 & 6 & 7 & 8 & 1 & 9 & 10 & 3 & 4 & 5 & 13 & 0 & 2 & 16 & 17 & 18 & 12 \\
        7 & 7 & 8 & 6 & 9 & 0 & 1 & 4 & 5 & 3 & 14 & 15 & 13 & 17 & 10 & 19 & 20 \\
        8 & 8 & 6 & 7 & 10 & 1 & 2 & 5 & 3 & 4 & 15 & 16 & 17 & 18 & 9 & 11 & 14 \\
        9 & 9 & 10 & 11 & 12 & 8 & 7 & 13 & 14 & 15 & 16 & 17 & 18 & 19 & 2 & 20 & 21 \\
        10 & 10 & 11 & 9 & 8 & 13 & 12 & 0 & 15 & 16 & 17 & 14 & 19 & 20 & 18 & 5 & 22 \\
        11 & 11 & 9 & 10 & 7 & 12 & 14 & 2 & 13 & 17 & 18 & 19 & 15 & 21 & 22 & 23 & 6 \\
        12 & 12 & 13 & 14 & 15 & 11 & 9 & 16 & 17 & 18 & 19 & 20 & 21 & 22 & 23 & 24 & 25 \\
        13 & 13 & 14 & 12 & 11 & 16 & 15 & 17 & 10 & 9 & 2 & 18 & 22 & 23 & 19 & 25 & 26 \\
        14 & 14 & 12 & 13 & 16 & 15 & 17 & 18 & 19 & 11 & 20 & 5 & 23 & 24 & 25 & 21 & 27 \\
        15 & 15 & 16 & 17 & 18 & 10 & 13 & 12 & 20 & 14 & 21 & 22 & 6 & 25 & 26 & 27 & 23
      \end{tabular}
    \end{center}
  \end{table}

  \begin{table}[h!]
    \scriptsize
    \caption{Grundy numbers for 9-Modular Wythoff's Game.}
    \begin{center}
      \begin{tabular}{c|cccccccccccccccc}
        & 0 & 1 & 2 & 3 & 4 & 5 & 6 & 7 & 8 & 9 & 10 & 11 & 12 & 13 & 14 & 15 \\
        \hline
        0 & 0 & 1 & 2 & 3 & 4 & 5 & 6 & 7 & 8 & 9 & 10 & 11 & 12 & 13 & 14 & 15 \\
        1 & 1 & 2 & 0 & 4 & 5 & 3 & 7 & 8 & 6 & 10 & 11 & 9 & 13 & 14 & 12 & 16 \\
        2 & 2 & 0 & 1 & 5 & 3 & 4 & 8 & 6 & 7 & 11 & 9 & 10 & 14 & 12 & 13 & 17 \\
        3 & 3 & 4 & 5 & 6 & 2 & 0 & 1 & 9 & 10 & 12 & 8 & 7 & 15 & 11 & 16 & 18 \\
        4 & 4 & 5 & 3 & 2 & 7 & 6 & 9 & 0 & 1 & 8 & 13 & 12 & 11 & 16 & 15 & 10 \\
        5 & 5 & 3 & 4 & 0 & 6 & 8 & 10 & 1 & 2 & 7 & 12 & 14 & 9 & 15 & 17 & 13 \\
        6 & 6 & 7 & 8 & 1 & 9 & 10 & 3 & 4 & 5 & 13 & 0 & 2 & 16 & 17 & 18 & 12 \\
        7 & 7 & 8 & 6 & 9 & 0 & 1 & 4 & 5 & 3 & 14 & 15 & 13 & 17 & 2 & 10 & 19 \\
        8 & 8 & 6 & 7 & 10 & 1 & 2 & 5 & 3 & 4 & 15 & 16 & 17 & 18 & 0 & 9 & 14 \\
        9 & 9 & 10 & 11 & 12 & 8 & 7 & 13 & 14 & 15 & 16 & 17 & 18 & 19 & 20 & 21 & 22 \\
        10 & 10 & 11 & 9 & 8 & 13 & 12 & 0 & 15 & 16 & 17 & 14 & 19 & 20 & 18 & 2 & 23 \\
        11 & 11 & 9 & 10 & 7 & 12 & 14 & 2 & 13 & 17 & 18 & 19 & 15 & 21 & 22 & 20 & 24 \\
        12 & 12 & 13 & 14 & 15 & 11 & 9 & 16 & 17 & 18 & 19 & 20 & 21 & 22 & 23 & 24 & 25 \\
        13 & 13 & 14 & 12 & 11 & 16 & 15 & 17 & 2 & 0 & 20 & 18 & 22 & 23 & 19 & 25 & 21 \\
        14 & 14 & 12 & 13 & 16 & 15 & 17 & 18 & 10 & 9 & 21 & 2 & 20 & 24 & 25 & 23 & 26 \\
        15 & 15 & 16 & 17 & 18 & 10 & 13 & 12 & 19 & 14 & 22 & 23 & 24 & 25 & 21 & 26 & 20
      \end{tabular}
    \end{center}
  \end{table}

\end{document}